\documentclass[reqno]{amsart}
\newcommand \datum {\hfill {\small{\red{Version of March 12, 2026}}}}%xxx
\usepackage{amssymb,latexsym}
\usepackage{amsmath}
\usepackage{hyperref}
\usepackage{url}
\usepackage{graphicx}
\usepackage[dvipsnames]{xcolor}
\usepackage{enumerate}
\usepackage{enumitem}
\usepackage{color}
\numberwithin{equation}{section}
\theoremstyle{plain}
 \newtheorem{theorem}{Theorem}[section]
 \newtheorem{lemma}[theorem]{Lemma}
 \newtheorem{corollary}[theorem]{Corollary}
\theoremstyle{definition}
\theoremstyle{remark}
\newtheorem{remark}[theorem]{Remark}
\newenvironment{caseinproof}[1]{\par\medskip \noindent\emph{#1.}\quad }{}

\providecommand{\logicbreak}{\par\vspace{0.5\baselineskip}}
\providecommand{\logicbreak}{\par\vspace{0.5\baselineskip}}
\newcommand \rrleq {\leq_\RR}
\newcommand \rrless {<_\RR}
\newcommand \rrcdot {\mathrel{\cdot_\RR}}
\newcommand \tups[1] {s^{(#1)}}
\newcommand \MExt[3] {\textup{MExt}(#1,#2; #3)}
\newcommand \oExt[1] {\textup{Ext}_1(#1)}
\newcommand \pnxt[1] {\textup{next}'(#1)}
\newcommand \nxt[1] {\textup{next}(#1)}
\newcommand \upe[1] {#1^{\textup{up}}}
\newcommand \dne[1] {#1_{\textup{dn}}}
\newcommand \Col[3]{\textup{Col}(#1;#2,#3)}
\newcommand \Skel[1]{\textup{Skel}(#1)}
\newcommand \LoCov[1]{\textup{LoCov}(#1)}
\newcommand \UpCov[1]{\textup{UpCov}(#1)}
\newcommand \pLoCov[2]{\textup{LoCov}_{#1}(#2)}
\newcommand \pUpCov[2]{\textup{UpCov}_{#1}(#2)}
\newcommand \nuc[1]{\textup{nuc}(#1)}
\newcommand \nlc[1]{\textup{nlc}(#1)}
\newcommand \nucp[2]{\textup{nuc}_{#1}(#2)}
\newcommand \nlcp[2]{\textup{nlc}_{#1}(#2)}
\newcommand \nothing[1] {}
\newcommand \upi[2] {#1^{(#2)}}
\newcommand \vupi[2] {\widehat{#1}^{(#2)}}
\newcommand \vari[1]{\mathcal #1}
\newcommand \Var[1]{\mathcal V(#1)}
\newcommand \varS {\vari S}
\newcommand \varW {\vari W}
\newcommand \Alll {\mathcal L_{\textup{all}}}
\newcommand \seql[1] {=_{\ast#1}}
\newcommand \seqast {=_{\ast}}
\newcommand \slleq[1] {\leq_{\ast#1}}
\newcommand \sleq {\leq_\ast}
\newcommand \Cls[1] {\textup{Cls}(#1)}
\newcommand \vSCD[1] {\textup{SCD}(#1)}
\newcommand \vwSCD {\vSCD{\varW}}
\newcommand \vsSCD {\vSCD{\varS}}
\newcommand \SCD {\textup{SCD}}
\newcommand \Acc[1] {\textup{Acc}(#1)}
\newcommand \ul[1] {\underline{#1}}
\newcommand \aref[1] {\textup{(a\ref{#1})}}
\newcommand \bref[1] {\textup{(b\ref{#1})}}
\renewcommand \phi{\varphi}
\newcommand \ZZ {{\mathbb Z}}
\newcommand \QQ {{\mathbb Q}}
\newcommand \RR {{\mathbb R}}
\newcommand \Nplu {{\mathbb N^+}}
\newcommand \Nnul {{\mathbb N_0}}
\newcommand{\tbf}{\textbf}% text bold font
\newcommand{\set}[1]{\{#1\}}% set 

\newcommand \cd[1]{\textup{cd}(#1)}
\newcommand \vcd[2]{\textup{cd}(#1,#2)}
\newcommand \vlnc [3] {\textup{lnc}(#1,#2,#3)} 
\newcommand \Con [1] {\textup{Con}(#1)}  
\newcommand \Jir [1] {\textup{Ji}(#1)}
\newcommand \Mir [1] {\textup{Mi}(#1)}
\newcommand \Jr [1] {\textup{Jr}(#1)}
\newcommand \Mr [1] {\textup{Mr}(#1)}
\newcommand \rno[1] {\textup{rno}(#1)}
\newcommand \Nar [1] {\textup{Nar}(#1)}
\newcommand \Core [1] {\textup{Cor}(#1)}
\newcommand \gsum {\mathrel{\dotplus}}
\newcommand \idl [1] {\textup{idl}(#1)}
\newcommand \fil [1] {\textup{fil}(#1)}
\newcommand \idla [2] {\textup{idl}_{#1}(#2)}
\newcommand \fila [2] {\textup{fil}_{#1}(#2)}
\newcommand \Idl [1] {\textup{Idl}(#1)}
\newcommand \con {\textup{con}}
\newcommand \Edge [1] {\textup{Edge}(#1)}

\newcommand \locov[1] {{#1}_\ast} %Technical Editor: we need the inner curly brackets "{" and "}"
\newcommand \quot[1] {``#1''}
\renewcommand \epsilon{\varepsilon}
\newcommand \doi[1] {\href{https://dx.doi.org/#1}{\nolinkurl{doi.org/#1}}}  
\newcommand \red[1]{{\textcolor{red}{#1}\color{black}}}

\begin{document}

\title[Accumulation points of congruence densities of finite lattices]
{Accumulation points of congruence densities of finite lattices}

\author[G.\ Cz\'edli]{G\'abor Cz\'edli}
\email{czedli@math.u-szeged.hu}
\urladdr{http://www.math.u-szeged.hu/~czedli/}
\address{University of Szeged, Bolyai Institute. 
Szeged, Aradi v\'ertan\'uk tere 1, HUNGARY 6720}

\begin{abstract} Let $\varW$ be a nontrivial variety of lattices, and let $L$ be a finite lattice in $\varW$. The \emph{congruence density} of $L$ with respect to $\varW$ is the number of congruences of $L$ divided by the maximum number of congruences of $|L|$-element lattices belonging to $\varW$. 
We prove that, with respect to the order and multiplication of the real numbers, the set $\vSCD\varW$ of congruence densities of finite members of $\varW$ as well as its topological closure are countably infinite dually well-ordered monoids. We also prove that the set of accumulation points of $\vSCD\varW$ is either a singleton or it is countably infinite; furthermore, it is a singleton if and only if $\varW$ is a subvariety of the variety of modular lattices. This gives a complicated characterization of modularity: a non-singleton lattice $K$ is modular if and only if $\vSCD{\Var K}$, where $\Var K$ denotes the variety generated by $K$, has only one accumulation point. The class $\vari S$ of semimodular lattices is not a variety, but $\vSCD{\vari S}$ is still meaningful; we prove that $\vSCD{\vari S}$ has exactly one accumulation point.
\end{abstract}

\dedicatory{Dedicated to George Gr\"atzer on the occasion of his ninetieth birthday, with appreciation and admiration, and in friendship}

\thanks{This research was supported by the National Research, Development and Innovation Fund of Hungary, under funding scheme K 138892.  
\hfill{\red{\tbf{\datum}}}}%xxx

%06B10(1980 now)Lattice ideals, congruence relations
%06C10(1980–now)Semimodular lattices, geometric lattices


\maketitle

\section{Introduction and our results}
The prerequisites for reading the paper are modest: only minimal familiarity with a few basic concepts from lattice theory or universal algebra is needed.

As usual, we denote a lattice by its base set; so $L$ will stand for $(L;\vee,\wedge)$. 
Unless explicitly stated otherwise or we take a class of lattices, every lattice in the paper is assumed to be \emph{finite}, even when this assumption is not repeated.
The number of congruences of a lattice $L$ is denoted by $|\Con L|$, where $\Con L$ stands for the \emph{congruence lattice} of $L$. Similarly, the \emph{size} (that is, the number of elements) of $L$ is denoted by $|L|$.
A lattice is \emph{nontrivial} if it has more than one element. A class $\varW$ of (not necessarily finite) lattices is \emph{nontrivial}, if it has a nontrivial lattice; we consider only nontrivial classes. Most of the classes we consider are \emph{varieties}, that is, they are defined by sets of equations or, equivalently by G.\ Birkhoff's classical HSP-theorem, they are closed under taking homomorphic images, sublattices, and direct products.

For positive integers $n$ and $k$ and a nontrivial class $\varW$ of lattices, if the set
$\{|\Con L| : L\in\varW$ and $|L|=n\}$ has at least $k$ elements, then the $k$th largest number in this set is denoted by $\vlnc {\varW} n k$. (The acronym is derived from the initial letters of \quot{largest number of congruences}.)
If the set in question has fewer than $k$ elements, then $\vlnc {\varW} n k$ is undefined. For a finite lattice $L$ in  $\varW$, the \emph{congruence density} $\vcd {\varW} L$ of $L$ \emph{with respect to} $\varW$ is defined as follows:
\begin{equation}
\vcd {\varW} L:= \frac{|\Con L|} {\vlnc{\varW}{|L|}1}.
\label{eq:cddef1}
\end{equation}
For example, if $\varW=\vari M$, the \emph{variety of modular lattices}, and $L=B_4$ is the 4-element Boolean lattice, then $\vcd {\vari M} {B_4}=1/2$, 
because there are exactly 
two\footnote{We do not distinguish isomorphic lattices in this paper.} 
$|B_4|$-element modular lattices, $B_4$ and the 4-element chain $C_4$, $|\Con{B_4}|=4$, $\vlnc{\vari M}{4}1=|\Con{C_4}|=8$, and so 
$\vcd {\vari M} {B_4}=4/8=1/2$.
Denoting by $\Alll$ the \emph{variety of all lattices}, the same argument shows that $\vcd \Alll {B_4}=1/2$.
We define
\begin{equation}
\vSCD {\varW}:=\set{\vcd {\varW}L: L \in\varW\text{ and }L \text{ is a finite}},
\label{eq:SCdef}
\end{equation}
the \emph{\ul set of \ul congruence \ul densities} of finite members of $\varW$.

The concepts defined in \eqref{eq:cddef1} and \eqref{eq:SCdef} are meaningful for general algebras other than lattices. For lattices, \eqref{eq:cddef1} simplifies by a result of Freese  \cite{freese}, see also \cite[Theorem 1(A)]{czg864}\footnote{\label{czg332}Alternatively, see Lemma 2.1 in the recent arXiv paper 
\quot{Lattices with congruence densities larger than 3/32}, 
\doi{10.48550/arXiv.2602.04321}.}, which asserts that for any 
class $\varW$ of lattices that contains all finite chains 
and for any positive integer $n$,  
$\vlnc{\vari K}{n}1=2^{n-1}$. 
Combining this equality with \eqref{eq:cddef1}, we obtain that if $\varW$ is a class of lattices containing all finite chains, and in particular if $\varW$ is a nontrivial variety of lattices, then for every finite lattice $L$ in $\varW$, 
\begin{equation}
\vcd {\varW}L:= |\Con L| / 2^{|L|-1}.
\label{eq:cddef2}
\end{equation}
To ease the notation and to harmonize with \cite{czg864,czgcdofS,czg332} , the variety $\Alll$ of all lattices is often dropped from our notations. Hence, for a finite lattice $L$, 
\begin{align}
\cd L&:=\vcd{\Alll}L=|\Con L|/2^{|L|-1} \text{ and }\label{eq:Alllnoa}
\\
\SCD&:=\vSCD\Alll=\set{\cd L:L\text{ is a finite lattice}}.\label{eq:Alllnob}
\end{align}

Congruence density is well related to some lattice constructions; see \eqref{eq:cdgsum}  and \eqref{eq:dismeXt}. Furthermore, \eqref{eq:Alllnoa} shows that the study of $|\Con L|$ is equivalent to the study of $\cd L$. 
This explains that the recent papers \cite{czg864}, \cite{czgcdofS}, and the one mentioned in Footnote \ref{czg332} focus on $\cd L$ rather than on $|\Con L|$.

As usual, $\Nplu$, $\Nnul$,  $\ZZ$, $\QQ$, and $\RR$ stand for the sets of positive integers, nonnegative integers, all integers, rational numbers, and real numbers, respectively. For real numbers $a<b$, the 
 interval $\set{x\in\RR: a<x\leq b}$ is denoted by $(a,b]_\RR$; the meanings of $(a,b)_\RR$,  $[a,b)_\RR$, and $[a,b]_\RR$ are similar. 
For a nontrivial variety $\varW$ of lattices, a number $r\in\RR$ is an \emph{accumulation point} of $\vwSCD$ if for all positive $\epsilon\in\RR$, the intersection
$(\vwSCD\setminus\set r) \cap(r-\epsilon,r+\epsilon)_\RR$ is nonempty. 
\begin{equation} 
\text{Let } \Acc{\vwSCD} \text{ denote the set of accumulation points of }\vwSCD. 
\label{eq:accscd}
\end{equation}
The topological closure of $\vwSCD$ is denoted by $\Cls{\vwSCD}$. That is,
\begin{equation}
\Cls{\vwSCD}=\vwSCD \cup \Acc\vwSCD.
\label{eq_clsW}
\end{equation}
Again, we may drop $\Alll$ from the notation, so
\begin{align}
\Acc\SCD&=\{r:r\text{ is an accumulation point of }\SCD\}\text{ and}
\label{eq:scdAccp}\\
\Cls\SCD&=\SCD\cup\Acc\SCD.
\label{eq:scdClsT}
\end{align}

For every nontrivial variety $\varW$ of lattices, it is not hard to verify---see Corollary \ref{corol:Vmod}\bref{cRl1}---that $\vwSCD$ is a countably infinite set. 
This follows from three known facts: (1) there are only countably many finite lattices, (2) $\cd{B_4}=1/2$, and (3) \cite[Lemma 2]{czg864}, which is cited here as \eqref{eq:cdgsum}. A countable set of real numbers can have continuously many accumulation points; for example, $\Acc{[0,1]_\RR\cap\QQ}=[0,1]_\RR$. However, we prove that $|\Acc\vwSCD|\neq 2^{\aleph_0}$; in fact, we prove more.

A lattice $L$ is \emph{semimodular} if for any $x,y,z\in L$ such that $y$ \emph{covers} $x$ (in notation, $x\prec y$), $y\vee z$ covers or equals $x\vee z$ (in notation, $x\vee z\preceq y\vee z$). 
The order and multiplication of the real numbers will be denoted by 
$\rrleq$ and $\rrcdot$, respectively. We also use these notations within subsets of $\RR$, and we may drop the subscripts when no ambiguity threatens.
For $X\subseteq\RR$, we say that $X$ is a \emph{dually well-ordered 
subset} of $\RR$ if every nonempty subset of $X$ has a greatest 
element, or equivalently, if there is no infinite strictly increasing 
sequence $a_{1}\rrless a_{2}\rrless a_{3}\rrless\dots$ in $X$.
Similarly, we say that a subset $X$ of $\RR$ is a \emph{dually 
well-ordered monoid with respect to $\RR$} if $X$ is a dually 
well-ordered subset of $\RR^{+}:=[0,\infty)_\RR=\set{r\in\RR:0\leq r}$, 
$1\in X$, and $a\rrcdot b\in X$ for all $a,b\in X$.
Sometimes, when the context is clear, we drop \quot{with respect to $\RR$}.
Note that such a monoid is an ordered monoid in the usual sense, since $\rrcdot$ is order-preserving on $\RR^+$.
As usual, the least infinite cardinal number is denoted by $\aleph_0$. 
With the notations \eqref{eq:SCdef}, \eqref{eq:Alllnob}, and \eqref{eq:accscd}, we can now state the main result of the paper.

\begin{theorem}\label{thm:main} 
The following three assertions hold.
\begin{enumerate}
\renewcommand{\labelenumi}{\textup{(a\theenumi)}}
\item\label{main1} %Reference with \aref{main1} 
The set $\SCD$ of congruence densities of all finite lattices is a dually well-ordered monoid with respect to $\RR$. 
\item\label{main2} %Reference with \aref{main2}
If a variety $\varW$ of lattices contains at least one (not necessarily finite) nonmodular lattice, then 
$|\Acc{\vwSCD}|=\aleph_0$. In particular, $\SCD$ has countably infinitely many accumulation points. 
\item\label{main3} %Reference with \aref{main3}
Letting $\varS$ denote the class of semimodular lattices, $\vsSCD$ is a submonoid of $\SCD$, and it has exactly one accumulation point: the real number $0$.
\end{enumerate}
\end{theorem}

Theorem \ref{thm:main} has several corollaries; these corollaries as well as the theorem  are proved, after several lemmas, in Section \ref{sect:rstprfs}. 
In addition to the notations used in the theorem, the corollaries also rely on those introduced in \eqref{eq_clsW}, \eqref{eq:scdAccp}, and \eqref{eq:scdClsT}.

\begin{corollary}\label{corol:Vmod} For every nontrivial variety $\varW$ of lattices, the following four assertions hold for $\vwSCD$ and, in particular, for $\SCD$.

\begin{enumerate}
\renewcommand{\labelenumi}{\textup{(b\theenumi)}}
\item\label{cRl1} %Reference with \bref{cRl1}$
$|\Cls\vwSCD|=|\vwSCD|=\aleph_0$.
\item\label{cRl2} %Reference with \bref{cRl2}
$\Cls\vwSCD$ is a dually well-ordered monoid with respect to $\RR$; $\vwSCD$ is a submonoid of it, and $\Acc\vwSCD$ is a subsemigroup of it. Moreover, the real number $0$ belongs to $\Acc\vwSCD$.
\item\label{cRl3} %Reference with \bref{cRl3}
$\vwSCD$ has either countably infinitely many accumulation points, or it has exactly one accumulation point. That is, $|\Acc\vwSCD|\in\set{1,\aleph_0}$.
\item\label{cRl4} %Reference with \bref{cRl4}
$\vwSCD$ has exactly one accumulation point if and only if $\varW$ is a subvariety of the variety $\vari M$ of modular lattices. 
\end{enumerate}
\end{corollary}

For a not necessarily finite lattice $L$, let $\Var L$ denote the \emph{variety generated} by $L$; it consists of all lattices satisfying the identities that hold in $L$. By G.\ Birkhoff's Theorem, $\Var L$ is the class of homomorphic images of sublattices of direct powers of $L$. The following two lemmas give complicated characterizations of modularity.

\begin{corollary}\label{corol:modL} Let $L$ be a not necessarily finite nontrivial lattice. Then $L$ is modular if and only if \,$\vSCD{\Var L}$ has exactly one (equivalently, at most one) accumulation point. 
\end{corollary} 

As usual, $\omega^\ast$ denotes the order type of the ordered set of negative integers. 

\begin{corollary}\label{corol:rescvmodL} Let $L$ be a not necessarily finite nontrivial lattice. Then $L$ is modular if and only if \,$\vSCD{\Var L}$ is of order type $\omega^\ast$. 
\end{corollary}

For $X\subseteq \RR$, a real number $r\in\RR$ is  a \emph{left accumulation point} of $X$ if for every positive $\epsilon\in\RR$, $X\cap(r-\epsilon,r)\neq\emptyset$.  Similarly, it is a  \emph{right accumulation point} of $X$ if 
$X\cap(r,r+\epsilon)\neq\emptyset$ for all positive $\epsilon\in\RR$.

\begin{corollary}\label{corol:scnd} Let $\varW$ be a nontrivial variety of lattices. Then every member of $\Acc\vwSCD$ is a right accumulation point but not a left accumulation point of $\vwSCD$. In particular, $\SCD$ has no left accumulation point.
\end{corollary}

\section{Some related work and possible prospects}

This paper is a straightforward continuation of the earlier papers that determine $\vlnc\Alll n k$ for some $n$ and $k$; see
Freese \cite{freese} for $k=1$, \cite{czg-lconl2} for $k=2$, 
Mure\c{s}an and Kulin \cite{muresankulin} for $n\geq 6$ and $k\in\set{3,4,5}$ (and for $(n,k)=(5,3)$), and  \cite{czg864,czg332} for $9\leq n$ and $k\leq n+4$ (and for some smaller $n$ if $k$ is small). 
There are also related papers in which congruences are replaced by other compatible relations like the unary subalgebra relation,  lattices are replaced by other algebras like semilattices, or some consequences of having many congruences or other compatible relations are studied. In addition to the papers mentioned previously, see, for example,
\cite{czg-83Sub}, \cite{czg-semlatmancon}, \cite{czg-manconplan}, \cite{czgnear83}, Ahmed and Horv\'ath \cite{delbrinKHE}, 
Kwuida and Mure\c{s}an \cite{kuwidamuresan},
 Zaja, Haje, and Ahmed \cite{zaya-haje-delbrin}, and the arXiv paper by Ahmed,  Salih, and Hale, \doi{10.48550/arXiv.2408.09595}.

These papers, together with the present one, indicate that for lattices and lattice-like algebras, as well as for certain other algebras, the study of various densities and their accumulation points might lead to further results.

Related to Theorem \ref{thm:main}\aref{main3}, note that there are many results on the congruence lattices of finite semimodular lattices. See, for example, Gr\"atzer \cite{GGstasect} and 
\cite[Chapters 10--11, 15, and 29]{GGbypicture}, the dual of Theorem 12 of Adaricheva, Freese and Nation \cite{kiraralphjb},
and more than a dozen additional papers from the Additional Bibliography\footnote{See \ \href{https://www.math.u-szeged.hu/\string~czedli/m/listak/publ-psml.pdf}{https://www.math.u-szeged.hu/\string~czedli/m/listak/publ-psml.pdf} \ for an update.} of the arXive paper\footnote{The journal version at
 \href{https://doi.org/10.7151/dmgaa.1416}{https://doi.org/10.7151/dmgaa.1416} 
does not contain the Additional Bibliography.}
\href{https://doi.org/10.48550/arXiv.2107.10202}{DOI 10.48550/arXiv.2107.10202}.

\section{Tools used in the paper}
This section recalls the necessary tools from the literature.
Two intervals, $[a_0,b_0]$ and $[a_1,b_1]$, of a lattice $L$ are \emph{transposed} if there is an $i\in\set{0,1}$ such that $b_i\wedge a_{1-i}=a_i$ and $b_i\vee a_{1-i}=b_{1-i}$. 
For elements $x,y\in L$, to fix the terminology, (1) $x\prec y$; (2) $y$ covers $x$; (3) $x$ is a lower cover of $y$; (4) $(x,y)$ is an edge, and (5) $[x,y]$ is a \emph{prime interval} are equivalent conditions.  For $a,b\in L$, the least congruence containing the pair $(a,b)$ (in other words, collapsing $a$ and $b$) is denoted by $\con(a,b)$ or $\con_L(a,b)$. The congruence $\con(a,b)$ is described in Gr\"atzer \cite[Theorem 230]{GGfound}---the description is credited to R.\ P.\ Dilworth. For a prime interval $[a,b]$, a more effective description of $\con(a,b)$ is presented in Gr\"atzer \cite{GGprimpers}. 
However, we need only the following well-known fact, which is straightforward per se and follows either from Gr\"atzer \cite[Theorem 230]{GGfound} or from Gr\"atzer \cite{GGprimpers}: for any two edges $(a,b)$ and $(c,d)$ in a lattice,
\begin{equation}
\label{eq:trsnpsdcon}
\text{if }[a,b]\text{ and }[c,d]\text{ are transposed, then }\con(a,b)=\con(c,d).
\end{equation}

An element $u\in L$ is \emph{join-irreducible} if it has exactly one lower cover. Denote by $\Jir L$ the set of join-irreducible elements of $L$. For $u\in\Jir L$, the unique lower cover of $u$ is denoted by $\locov u$. The set of elements with exactly one (upper) cover is denoted by 
$\Mir L$, its members are the \emph{meet-irreducible elements}. For $a,b\in\Jir L$, 
\begin{align}
\label{eq:dkJJsleq}
&\text{let }a\sleq b\text{ or } a\slleq L b\text{ denote that }\con(\locov a, a) \leq \con(\locov b, b),\text{ and}\\
\label{eq:dkKKseq}
&\text{let }a\seqast b \text{ or } a\seql L b \text{ stand for }\con(\locov a, a) = \con(\locov b, b).
\end{align}

Armed with the reflexive and transitive relation defined in \eqref{eq:dkJJsleq}, $(\Jir L;\slleq L)$
is a quasiordered set. A subset $X$ of $\Jir L$ is an \emph{ideal} of $(\Jir L;\slleq L)$ if for any $x\in X$ and $y\in \Jir L$, $y\slleq L x$ implies $y\in X$. Let $\Idl{\Jir L;\slleq L}$ denote the set of ideals of $(\Jir L;\slleq L)$. Then  $(\Idl{\Jir L;\slleq L};\cup,\cap)$ is a distributive lattice. 
By (2.1), (2.2), and (2.7) of \cite{czg864}\footnote{Alternatively, by (3.8) and (3.13) in the paper mentioned in Footnote \ref{czg332}, or  by known facts: $\Con L$ is distributive; $D\cong(\Idl{\Jir D;\leq}$ for each finite distributive lattice $D$; and $\Jir{\Con L}=\set{\con(\locov u,u):u\in\Jir L}$ for every finite lattice $L$.}, we have that $\Con L\cong ( \Idl{\Jir L;\slleq L};\cup,\cap)$. Hence
\begin{equation}
|\Con L| = |\Idl{\Jir L;\slleq L}|.
\label{eq:conLIdl}
\end{equation}
Quite often, \eqref{eq:conLIdl} can be improved as follows. Assume that $B$ is a subset of $\Jir L$ such that, using the notation introduced in \eqref{eq:dkKKseq}, 
\begin{equation}
\text{for each }x\in\Jir L,\text{ there is a }b\in B
\text{ such that }x\seql L b.
\label{eq:condBJirL}
\end{equation}
It follows straightforwardly from \eqref{eq:dkJJsleq}, \eqref{eq:dkKKseq}, and the sentence preceding \eqref{eq:conLIdl}\footnote{Alternatively, it also follows from (3.8) and (3.13) of the paper mentioned in Footnote \ref{czg332}.} 
that \eqref{eq:condBJirL} implies $\Con L\cong\Idl{B;\slleq L}$. Therefore, 
\begin{equation}
\label{eq:smlBmpl}
\text{if a subset }B\text{ of }\Jir L\text{ satisfies \eqref{eq:condBJirL}, then }|\Con L|\leq 2^{|B|}.
\end{equation}

The elements of $L$ with more then one lower cover are \emph{join-reducible}; their set is denoted by $\Jr L$. Similarly, $\Mr L$ stands for the set of \emph{meet-reducible} elements, which are defined dually.  Since $0\notin\Jir L\cup \Jr L$, we have $|L|=1+|\Jir L|+|\Jr L|$. Therefore, using \eqref{eq:smlBmpl}, the fact that $\Jir L$ in place of $B$ satisfies  \eqref{eq:condBJirL}, and duality,  
\begin{equation}
\cd L= \frac{|\Con L|}{2^{|L|-1}}\leq \frac{2^{|\Jir L|}}{2^{|\Jir L|+|\Jr L|}} = 2^{-|\Jr L|}\text{ and }\cd L\leq 2^{-|\Mr L|}.
\label{eq:mncgrdlcd}
\end{equation}

For finite lattices $K$ and $M$, we obtain the (diagram of) the \emph{glued sum} $K\gsum M$ by putting the diagram of $M$ atop the diagram of $K$ and identifying the top $1_{K}$ of $K$ with the bottom $0_{M}$ of $M$.  
Observe that $K\gsum M$ is isomorphic to the sublattice $(K\times\set{0_{M}})\cup (\set{1_{K}}\times M )$ of the direct product $K\times M$. Therefore, for every variety $\varW$ of lattices and any two finite lattices $K$ and $M$,
\begin{equation}
\text{if \ }K,M\in\varW\text{, \ then \ }K\gsum M\in\varW.
\label{eq:varWgsumclosed}
\end{equation}
Lemma 2 of \cite{czg864} asserts that for finite lattices $L_1,\dots,L_t$, 
\begin{equation}
\cd{L_1\gsum\cdots \gsum L_t}=\cd{L_1}\dots \cd{L_t}.
\label{eq:cdgsum}
\end{equation}
In particular, since $\cd D=1$ for any finite chain $D$ by Freese \cite{freese} or, more directly, by \cite[Theorem 1(A)]{czg864}, we have for a finite chain $C$ and a finite lattice $L$ that
\begin{equation}
\cd{L\gsum C}=\cd{C\gsum L}=\cd L.
\label{eq:chngsum}
\end{equation}

If $S$ is a sublattice of a finite lattice $L$ and there are sublattices $S_0, \dots, S_t$ of $L$ such that $S_0=S$, $S_t=L$, and for all $i\in\set{1,\dots,t}$, $S_{i-1}\subseteq S_i$ and $|S_i\setminus S_{i-1}|=1$, then we say that $L$ is a \emph{dismantlable extension of $S$}. It is proved in \cite{czgcdofS} that for finite lattices $L$ and $S$,
\begin{equation}
\text{if }L\text{ is a dismantlable extension of }S,\text{ then }\cd L\leq \cd S.
\label{eq:dismeXt}
\end{equation}

For a finite lattice $L$, an element of $L$ is called a \emph{narrows} if it is comparable with every element of $L$. The set of narrows of $L$ is denoted by $\Nar L$. If $(a,b)$ is an edge of $L$ (that is, if $a\prec b$) and $a,b\in\Nar L$, then $(a,b)$ is called a \emph{gluing edge} of $L$. The remaining edges are called \emph{non-gluing edges}.
For an element $w$ in $L$, let $\set{x\in L: x\leq w}$ and $\set{x\in L: x\geq w}$ be denoted by $\idl w$ and $\fil w$, or, to indicate the lattice, by $\idla L w$ and $\fila L w$, respectively. Now if $(a,b)$ is a gluing edge of $L$, then $L$ decomposes as $L=\idl a\gsum [a,b]\gsum \fil b$. Denote 
$\idl a\gsum \fil b$ by $\Col L a b$; we say that $\Col L a b$ is obtained from $L$ by \emph{collapsing the gluing edge} $(a,b)$. 
The point is that the middle glued-sum summand $[a,b]$ is the two-element chain. Therefore, using either \eqref{eq:chngsum} repeatedly or \eqref{eq:cdgsum}, we obtain that
\begin{equation}
\begin{aligned}
\cd L&=\cd{\idl a\gsum [a,b]\gsum \fil b}=
\cd{\idl a}\cdot\cd{[a,b]}\cdot \cd{\fil b} \cr
&=\cd{\idl a}\cdot \cd{\fil b} =\cd{\idl a\gsum \fil b} = \cd {\Col L a b}.
\end{aligned}
\label{eq:cdColLab}
\end{equation}
After collapsing all gluing edges, one by one in an arbitrary order, we obtain the \emph{core} of $L$, which is denoted by $\Core L$. 
The following lemma follows from \eqref{eq:cdColLab}, and it is implicit in  \cite[Lemma 3]{czg864}\footnote{It is similarly implicit in (3.1) of the paper mentioned in Footnote \ref{czg332}.}.

\begin{lemma}\label{lemma:Corelem}
For each finite lattice $L$, the edges of $\Core L$ are exactly the non-gluing edges of $L$. Furthermore, $\cd L=\cd{\Core L}$.
\end{lemma}

The \emph{one-point extension} $\oExt{L; (u,v)}$ 
of a finite lattice $L$ at an edge $(u,v)$ of $L$ is defined to be $L\cup \set{x}$ such that $x\notin L$, $u\prec x\prec v$ in $\oExt{L; (u,v)}$, and $x$ is both meet- and join-irreducible in $\oExt{L; (u,v)}$. In other words, we obtain the diagram of $\oExt{L; (u,v)}$ by inserting a new vertex on the edge from $u$ to $v$ in the diagram of $L$. That is, the old edge $(u,v)$ is replaced by two new edges, $(u,x)$ and $(x,v)$ in $\oExt{L; (u,v)}$. As the diagram determines the lattice, the following folklore lemma, taken from Gr\"atzer \cite[Lemma 290]{GGfound}, is essentially immediate.

\begin{lemma}[Gr\"atzer \cite{GGfound}]\label{lemma:opExt}
 For each edge $(u,v)$ of a finite lattice $L$, $\oExt{L; (u,v)}$ is a lattice and $L$ is a sublattice of it.
Furthermore, $\oExt{L; (u,v)}$ is a dismantlable extension of $L$.
\end{lemma}

\section{Six preparatory lemmas}

Let $e_1=(u_1,v_1)$, \dots, $e_t=(u_t,v_t)$ be an enumeration of the edges of a finite lattice $L$. Denote this enumeration by $\pi$. Let $s=(s_1,\dots,s_t)\in\Nnul^t$ be a $t$-tuple of nonnegative integers. Repeat the one-point extension $s_1+\dots+s_t$ times so that first we insert $s_1$
new elements (in other words, vertices) on the edge $e_1$ in the diagram of $L$, then $s_2$ new elements on the edge $e_2$, \dots, and finally $s_t$ new elements on  the edge $e_t$. Let 
$\MExt L \pi s$ 
 denote the poset that we obtain in this way; we call it the \emph{multi-point extension} of $L$ with parameters $\pi$ and $s$. For $r=(r_1,\dots,r_t)$ and $s=(s_1,\dots,s_t)$ in $\Nnul ^t$, $r\leq s$ means that $r_i\leq s_i$ for all $i\in\set{1,\dots,t}$.

\begin{lemma}\label{lemma:MExt}
$\MExt L \pi s$ defined above is a finite lattice, and it is a dismantlable extension of $L$. Furthermore, if $r,s\in\Nnul^t$ with $r\leq s$, then $\MExt L\pi  s$ is a dismantlable extension of $\MExt L\pi  r$.  
\end{lemma}

\begin{proof}
To add $s_1$ new vertices on the edge $e_1$, we apply the one-point extension $s_1$ times. So let $L_0:=L$ and $w_{1,0}:=u_1$. For $i\in\set{1,\dots, s_1}$, let 
$L_i:=\oExt{L_{i-1};(w_{1, i-1}, v_1)}$, and let $w_{1,i}$ be the new vertex. That is $w_{1,i}$ is the only element of $L_i\setminus L_{i-1}$.  

In $L_{s_1}$, all the $s_1$ new vertices appear on the (original) edge of $L$. We proceed similarly for the second edge of $L$. Namely, let $w_{2,0}:=u_2$. For 
$i\in\set{1,\dots, s_2}$, let 
$L_{s_1+i}:=\oExt{L_{s_1+i-1};(w_{2, i-1}, v_2)}$, and let $w_{2,i}$ be the new vertex.

Then we continue with the edge $e_3$, and so on. Finally, we arrive at $L_{s_1+\dots+s_t}=\MExt L\pi  s$. 
It follows from a repeated use of Lemma \ref{lemma:opExt} that, for all $j\in\set{1,2,\dots, s_1+\dots+s_t}$,  $L_j$ and $\MExt L\pi s$ are lattices and $L_j$ is a \quot{one-step} dismantlable extension of $L_{j-1}$.

The sublattices $L_j$ witness that $\MExt L\pi  s$ is a dismantlable extension of $L$. 
Clearly, if $r \leq s$, then we can reach $\MExt L\pi  s$ from $\MExt L\pi r$ by one-point extensions (in zero steps when $r=s$).
Therefore $\MExt L\pi s$ is a dismantlable extension of $\MExt L\pi  r$, completing the proof of Lemma~\ref{lemma:MExt}.
\end{proof}

\begin{remark}%\label{rem:Extnts} 
The lattices $L$ and $\MExt L\pi s$ do not determine $\pi$ and $s$. For example, if $L=B_4$, the four-element Boolean lattice, then any choice of $\pi$ and every $s\in\{(1,0,0,0)$, $(0,1,0,0)$, $(0,0,1,0)$, $(0,0,0,1)\}$ yield $N_5$, the five-element nonmodular lattice.
\end{remark}

For $k\in\Nplu$ with $k\geq 3$, let $M_k$ denote the $(k+2)$-element lattice that consists of $k$ pairwise incomparable elements $a_1,\dots, a_k$, a bottom element $u$, and a top element $v$. In other words, $M_k$ is the $(k+2)$-element modular lattice of length~$2$.

\begin{lemma}\label{lemma:Mk} For $3\leq k\in\Nplu$, let the $M_k$ described above be a sublattice of a finite lattice $L$ such that $u\prec_L a_i$ for $i\in\set{1,\dots,k}$. That is, the lower $k$ edges of $M_k$ are also edges of $L$, but the upper $k$ edges need not be. Then $\cd L\leq 2^{-k}$.
\end{lemma}

\begin{proof}
For $i\in\set{1,\dots,k}$, pick a minimal element $b_i\in\idl{a_i}\setminus \idl u$. 
If $b_i$ had two distinct lower covers $x$ and $y$, then $x$ and $y$ would belong to $\idl u$ by the minimality of $b_i$, and therefore $x\vee y$ would also belong to $\idl u$, contradicting $x\vee y=b_i\notin\idl u$.
Hence $b_i\in\Jir L$. Using $u \prec a_i$ and $\locov {b_i}\prec b_i$ together with $b_i\nleq u$ and $\locov{b_i}\leq u$, it follows that $[\locov{b_i},b_i]$ and $[u,a_i]$ are transposed intervals.
Hence \eqref{eq:trsnpsdcon} and the equality $a_i=b_i\vee u$ yield, respectively, that
\begin{equation}
\con_L(\locov{b_i},b_i)=\con_L(u,a_i)\text{ for }i\in\set{1,\dots,k}\text{, and }|\set{b_1,\dots,b_k}|=k.
\label{eq:ftnddistinct}
\end{equation}
Next, for $i\in\set{1,\dots,k}$, pick an element $c_i\in L$ such that $a_i\leq c_i\prec v$. For distinct $i,j\in\set{1,\dots,k}$, we have $a_j\nleq c_i$, since otherwise $v=a_i\vee a_j\leq c_i\vee c_i=c_i$ would be a contradiction. Hence 
$c_i<a_j\vee c_i\leq v$ and  $c_i\prec v$ yield $a_j\vee c_i=v$. Similarly, $u\leq a_j\wedge c_i<a_j$ and $u\prec a_j$ imply $a_j\wedge c_i=u$. Thus $[u,a_j]$ and $[c_i,v]$ are transposed intervals. Therefore \eqref{eq:trsnpsdcon} gives that 
\begin{equation}
\text{for any two distinct }i,j\in\set{1,\dots,k},
\text{ \ }\con_L(u,a_j)=\con_L(c_i,v).
\label{eq:conuajciv}
\end{equation}
For any two distinct $i,j\in\set{1,\dots,k}$, we can pick a $t\in\set{1,\dots,k}\setminus\set{i,j}$. Applying \eqref{eq:conuajciv} first for $i$ and $t$ and then for $t$ and $j$, we obtain that $\con(u,a_i)=\con(c_t,v)=\con(u,a_j)$. Hence $\con_L(u,a_1)=\dots=\con_L(u,a_k)$. Combining this chain of equalities with \eqref{eq:dkKKseq} and the first half of \eqref{eq:ftnddistinct}, it follows that $b_1\seql L b_2\seql L\dots \seql L b_k$. Hence $B:=\Jir L\setminus \set{b_2,\dots,b_k}$ satisfies \eqref{eq:condBJirL}. Furthermore, by the second half of \eqref{eq:ftnddistinct}, $v\in\Jr L$, and the equality $|\Jir L|=|L|-1-|\Jr L|$, we have that $|\Jir L|\leq |L|-1-1=|L|-2$ and that
\begin{equation*}
|B|=|\Jir L|-(k-1)\leq |L|-2-(k-1)=|L|-1 -k.
\end{equation*} 
This inequality and \eqref{eq:smlBmpl} imply that  $|\Con L|\leq 2^{|B|} = 2^{|L|-1 -k}$. Finally, dividing this inequality by $2^{|L|-1}$, it follows that $\cd L\leq 2^{-k}$, which completes the proof of Lemma \ref{lemma:Mk}.
\end{proof}

For positive integers $c,k_1,\dots,k_c$, let $R(k_1,\dots,k_c)$ be the smallest number $n$ such that, no matter how we color the edges of the $n$-element complete (simple, undirected) graph $K_n$ with $c$ colors $\gamma_1,\dots,\gamma_c$\footnote{Each edge of $K_n$ has exactly one color, and some colors $\gamma_i$ may remain unused.}, there is an $i\in\{1,\dots,c\}$ and a $k_i$-element complete subgraph of $K_n$ all of whose edges have color $\gamma_i$.
By Ramsey's theorem \cite{Ramsey}, 
\begin{equation}
R(k_1,\dots,k_c)\text{ is a positive integer; in particular, it exists.}
\label{eq:Ramsey}
\end{equation}
For an integer $k\geq 3$, we introduce the following notation, with $k-1$ copies of $k$ on the right: 
\begin{equation}
R_{k-1}(k):=R(\underbrace{k,k,\dots,k}_{k-1}).
\label{eq:notaRkkm}
\end{equation}

\begin{lemma}\label{lemma:ramseyMt} 
For any integer $k\geq 3$ and every finite lattice $L$, if an element of $L$ has at least $R_{k-1}(k)$ covers, then 
$\cd L\leq 2^{-k}$.
\end{lemma}

We suspect that the lemma holds under far weaker assumptions---perhaps even with only $k$ covers. It is, however, sufficient for our purposes and easy to prove.

\begin{proof} Let $L$ be as in the lemma, and denote $R_{k-1}(k)$ by $t$. Then $t\in\Nplu$ by \eqref{eq:Ramsey}. Obviously, $k\leq t$. Pick an element $u\in L$ with at least $t$ covers, and let $a_1,\dots, a_t$ be pairwise distinct covers of $u$. Let $H:=\set{a_i\vee a_j: 1\leq i<j\leq t}$; it is a subset of $\Jr L$.  If $|H|\geq k$, then $|\Jr L|\geq k$ and \eqref{eq:mncgrdlcd} yield the required equality $\cd L\leq 2^{-k}$. Hence, we can assume that $|H|<k$. Thus there are pairwise distinct elements $h_1,\dots,h_{k-1}\in L$ such that 
$H\subseteq\set{h_1,\dots,h_{k-1}}$. Let $K_t$ be the $t$-element complete graph with vertex set $\set{a_1,\dots,a_t}$. We color the edges of $K_t$ with $h_1,\dots,h_{k-1}$ as follows: for $1\leq i<j\leq t$, the edge $(a_i,a_j)$ has color $a_i\vee a_j$.  
Note that unless $|H|=k-1$, some colors are unused. By the definition of the Ramsey number $t=R_{k-1}(k)$, there is an $s\in\set{1,\dots,k-1}$ and a $k$-element \quot{monochromatic} complete subgraph $\set{a_{p_1},\dots, a_{p_k}}$ such that for all $1\leq i<j\leq k$, the edge $(a_{p_i},a_{p_j})$ has color $h_s$. That is, $a_{p_i}\vee a_{p_j}=h_s$.
Therefore, with $(u,h_s, a_{p_1},\dots, a_{p_k})$ in place of $(u,v,a_1,\dots,a_k)$,  Lemma \ref{lemma:Mk} applies, and we conclude that $\cd L\leq 2^{-k}$. The proof of Lemma \ref{lemma:ramseyMt} is complete.
\end{proof}

For $k\in\Nplu$ and $k$-tuples $a=(a_1,\dots, a_k)$ and $b=(b_1,\dots, b_k)$, let $a\leq b$ denote that $a_i\leq b_i$ for all $i\in\set{1,\dots,k}$.

\begin{lemma}\label{lemma:woQo} Given an integer $k\in \Nplu$, let $\upi s 1$, 
$\upi s 2$, $\upi s 3$, \dots be an infinite sequence of $k$-tuples of nonnegative integers.  Then there are $i,j\in\Nplu$ such that $i<j$ and $\upi s i\leq \upi s j$. 
\end{lemma}

\begin{proof} We prove the lemma by induction on $k$. The statement is trivial for $k=1$. So assume that $k\geq 2$ and that the lemma holds for all smaller values. For the sake of contradiction, we assume that there is an infinite sequence 
\begin{equation}
\upi s 1=(\upi s 1_1,\dots,\upi s 1_k),\  \upi s 2=(\upi s 2_1,\dots,\upi s 2_k), \  \upi s 3=(\upi s 3_1,\dots,\upi s 3_k),\  \dots
\label{eq:ssrznClb}
\end{equation}
of nonnegative integers such that 
\begin{equation}
\upi s i\nleq \upi s j\text{ for all }i,j\text{ with }i<j.
\label{eq:thfltNclB}
\end{equation}

For $j\in\set{1,\dots,k}$, define $T_j:=\set{\upi s i_j: i\in\Nplu}$. We claim that at least one out of $T_1$, \dots, $T_k$ is infinite. Suppose the contrary. Then, for all $i\in\Nplu$,  the $k$-tuple  $\upi s i$ belongs to the finite
set $T_1\times T_2\times\dots \times T_k$. Hence there are $i,j\in\Nplu$ such that $i<j$ and $\upi s i=\upi s j$, contradicting \eqref{eq:thfltNclB}. Therefore, $T_j$ is infinite for some $j\in\set{1,\dots, k}$. By symmetry, we may assume that this $j$ is $k$, that is, $T_k$ is infinite.
To define a subsequence of the sequence in \eqref{eq:ssrznClb}, let $m_1:=1$. Since $T_k$ is infinite, there is a smallest $m_2\in\Nplu$ such that $m_1<m_2$ and $\upi s{m_1}_k<\upi s{m_2}_k$. Again, since $T_k$ is infinite, there is a smallest $m_3\in\Nplu$ such that $m_2<m_3$ and $\upi s{m_2}_k<\upi s{m_3}_k$. And so on; if $m_j$ is already defined, then $m_{j+1}\in\Nplu$ is the smallest integer such that $m_j<m_{j+1}$ and $\upi s {m_j}_k < \upi s {m_{j+1}}_k$. 
To ease the notation, we may replace the original sequence by the subsequence $\upi s{m_1}$, $\upi s{m_2}$, $\upi s{m_3}$, \dots; \eqref{eq:thfltNclB} remains valid. 
Thus, due to this replacement, we may assume that in addition to \eqref{eq:thfltNclB}, the sequence $\upi s 1$, $\upi s 2$, $\upi s 3$, \dots {} has the additional property
\begin{equation}
\upi s 1_k<\upi s 2_k < \upi s 3_k < \upi s 4_k<\dots
\label{eq:pdDtnlpRt}
\end{equation} 
By letting $\vupi s i:=(\upi s i_1,\dots, \upi s i_{k-1})$, we obtain a  sequence $\vupi s 1$, $\vupi s 2$, $\vupi s 3$, \dots{} of $(k-1)$-tuples.  Combining \eqref{eq:thfltNclB} with \eqref{eq:pdDtnlpRt}, it follows that this sequence of  $(k-1)$-tuples also satisfies \eqref{eq:thfltNclB}. This contradicts the induction hypothesis and completes the proof of Lemma \ref{lemma:woQo}.
\end{proof}

Note that the proof above was motivated by the well-known fact that the direct power $(\Nnul^k;\leq)$
has neither an infinite strictly decreasing sequence nor an infinite antichain; an \emph{antichain} is a poset consisting of pairwise incomparable elements. 

A chain $H=(H;\leq)$ is a \emph{well-ordered set} if each of its nonempty subsets has a (uniquely defined) smallest element. For a subset $H$ of $\RR$, we say that $H$ is a \emph{well-ordered subset} of $\RR$ if $(H;\leq)$, where \quot{$\leq$} is the restriction of the usual ordering $\leq_\RR$ of the real numbers, is a well-ordered set.
The following easy lemma belongs to the folklore; its 
parts\footnote{We found two examples; see  
\href{https://people.math.harvard.edu/\string~elkies/Misc/sol6.html}{https://people.math.harvard.edu/\string~elkies/Misc/sol6.html} and 
\href{
https://math.stackexchange.com/questions/1983425/there-is-no-well-ordered-uncountable-set-of-real-numbers}{https://math.stackexchange.com/questions/1983425/there-is-no-well-ordered-uncountable-set-of-real-numbers.}} are at exercise level; we include a short proof for completeness.

\begin{lemma}\label{lemma:wLlordH} 
Assume that $H\neq\emptyset$ is a well-ordered subset of $\RR$. Then $|H|\leq |\Cls H|\leq\aleph_0$, and $\Cls H$ is also a well-ordered set. Furthermore, every element of $\Acc H$ is a left accumulation point but not a right accumulation point of $H$.
\end{lemma}

\begin{proof} For $r\in \RR$, let $\nxt r\in H$ be the smallest element $h\in H$ such that $r<h$; if there is no such element, then $\nxt r:=\infty$. 
%(The acronym in $\nxt  r$ comes from \emph{next}.)% 
For any $r\in\RR$,  $(r,\nxt r)_\RR\cap H=\emptyset$, whence $r$ is not a right accumulation point. Therefore $H$ has no right accumulation point.

Denoting the smallest element of $H$ by $h_0$, observe that $h_0\leq x$ for every $x\in\Acc H$. 
To show that $\Cls H=H\cup\Acc H$ is also well-ordered, assume that $\emptyset\neq X\subseteq \Cls H$. For the sake of contradiction, suppose that $X$ has no smallest element. 
Let $u\in \RR$ denote the infimum of $X$; since $h_0\leq x$ holds for all $x\in X$, this infimum exists. As $X$ has no smallest element, $u\notin X$. For each $n\in\Nplu$, pick an element $p_n\in(u,u+1/n)_\RR\cap X$, and define $h_n\in H$ as follows. If $p_n\in H$, then $h_n:=p_n$. 
Otherwise $p_n\in \Acc H$, and we pick $h_n\in H$ from $(u,u+1/n)_\RR$, which is a neighborhood of $p_n$. Since $h_n\in(u,u+1/n)_\RR$ in both cases, $u$ is a right accumulation point of $H$, but this has been excluded. Therefore, $\Cls H$ is a well-ordered subset of $\RR$.

For each $r\in\RR$, we define $\pnxt r\in\Cls H$ as the smallest element $c\in\Cls H$ with $r<c$; we let $\pnxt r=\infty$ if no such $c$ exists.
For each $x\in \Cls H$, pick a rational number $q_x\in (x,\pnxt x)$. If $x,y\in\Cls H$ with $x<y$, then $q_x<\pnxt x\leq y<q_y$, whence the function $x\mapsto q_x$ from $\Cls H$ to $\QQ$ is injective. Thus $|H|\leq|\Cls H|\leq |\QQ|= \aleph_0$, completing the proof of Lemma \ref{lemma:wLlordH}.
\end{proof}

\section{Skeleton sublattice}
For an element $w$ of a finite lattice $L$, let the following
\begin{align*}
\LoCov w&=\pLoCov L w:=\set{x \in L: x\prec w}\text{ and}
\cr
\UpCov w&=\pUpCov L w:=\set{x \in L: w\prec x}
\end{align*}
denote the \emph{set of lower covers} and the \emph{set of (upper) covers} of $w$, respectively.
The \emph{number $|\LoCov w|$ of lower covers} and the \emph{number $|\UpCov w|$ of (upper) covers} of $w$ are denoted by $\nlc w$ and $\nuc w$, respectively. We may also write $\nlcp L w$ and $\nucp L w$ to indicate the lattice $L$.

The \emph{skeleton sublattice} $\Skel L$ of a finite lattice $L$ is defined by
\begin{equation}
\Skel L:=\Jr L\cup\Mr L\cup \bigcup_{x\in\Jr L}\LoCov x\cup \bigcup_{x\in\Mr L}\UpCov x .
\label{eq:szlSklSubL}
\end{equation}
This terminology is explained by the fact that 
\begin{equation}
\Skel L\text{ is a sublattice of }L.
\label{eq:SkelLsublat}
\end{equation}
Indeed, if $x,y\in \Skel L$ are comparable, then $x\vee y$ and $x\wedge y$ both belong to $\{x,y\}\subseteq \Skel L$. If $x,y\in \Skel L$ are incomparable, then $x\wedge y\in\Mr L\subseteq \Skel L$ and $x\vee y\in\Jr L\subseteq \Skel L$. Hence \eqref{eq:SkelLsublat} holds.

We define the \emph{reducibility number} of a finite lattice $L$ by
\begin{align}
\rno L:=\sum_{x\in\Jr L}(\nlc x+1) + \sum_{x\in\Mr L}(\nuc x +1).
\label{eq:SumRSdef}
%\label{eq:mRlmrnKbdf} 
\end{align}

The following lemma is a straightforward consequence of \eqref{eq:szlSklSubL} and \eqref{eq:SumRSdef}.

\begin{lemma}\label{lemma:ineqLSumRS} For every finite lattice $L$, we have $|\Skel L|\leq \rno L$.
\end{lemma}

For a finite lattice $L$ with $1\in\Jr L$ and $0\in\Mr L$, 
we have $0,1\in\Skel L$. Therefore, for each $x\in L$, the meet and the join in the following definition are nonempty; $\bigwedge X$ is the same as $\bigwedge_{u\in X}u$, and dually. We define
\begin{equation*}
\upe x:=\bigwedge \bigl(\Skel L\cap\fila L x \bigr)
\text{ and }\dne x:=\bigvee \bigl(\Skel L\cap\idla L x \bigr).
\end{equation*}
Clearly, $\upe x$ is the smallest element in $\Skel L\cap \fila L x$, and dually. Furthermore,
\begin{align}
\upe x, \dne x\in\Skel L,\text{ }\dne x\leq x\leq \upe x
\text{, and }x\in\Skel L\iff \dne x=\upe x.
\label{eq:rSzjbGnxcj}
\end{align}

For a lattice $L$, let $\Edge L:=\set{(u,v)\in L: u\prec_L v}$ denote the \emph{edge set} of $L$.

\begin{lemma}\label{lemma:svznfmsT} 
If $L$ is a finite lattice with $0\in\Mr L$ and $1\in\Jr L$, then 
\begin{equation}
L=\bigcup_{(u,v)\in\Edge{\Skel L}} [u,v]_L 
\label{eq:lLlflTnS}
\end{equation}
and for any edge $(u,v)\in\Edge{\Skel L}$, the interval $[u,v]_L$ is a chain.
\end{lemma}

\begin{proof} Denote the right-hand side of \eqref{eq:lLlflTnS} by $T$. For brevity, let $S$ stand for $\Skel L$.  Clearly, $L\supseteq T$ and $S\subseteq T$. So to verify \eqref{eq:lLlflTnS}, it suffices to show that for every $x\in L\setminus S$, we have $x\in T$. For an element $x\in L\setminus S$, let $u:=\dne x$ and $v:=\upe x$. Then, by \eqref{eq:rSzjbGnxcj} and $x\notin S$, we have $u<x<v$. So, to show that $x\in T$, it suffices to show that $u\prec_S v$. Suppose the contrary, and pick an element $w\in S$ such that 
$u<w< v$. 
If we had $x\leq  w$, then $v=\upe x\leq w$ would contradict $w < v$. 
Similarly, $w\leq x$ would
yield $w\leq \dne x =u< w$, a contradiction again. Thus $w\parallel x$ (notation for incomparability), and $x< w\vee x\in\Jr L\subseteq S$. 
Hence $v=\upe x\leq w\vee x$. But $w\vee x\leq v$, whence $v=w\vee x\in \Jr L$, and thereby $\pLoCov L v\subseteq S$.
Since $x<v$, there exists a $y\in\pLoCov L v$ such that $x\leq y$. Using $y\in S$, we obtain that  $v=\upe x\leq y\prec_L v$. This is a contradiction, proving $u\prec_S v$ and  \eqref{eq:lLlflTnS}.

Next, aiming at a contradiction, suppose that $u\prec_S v$ but $[u,v]_L$ is not a chain. Pick $x,y\in [u,v]_L$ such that $x\parallel y$. 
Then $u<x\vee y\leq v$, $x\vee y\in \Jr L\subseteq S$, and $u\prec_S v$ imply $v=x\vee y\in \Jr L$. Hence $\pLoCov L v\subseteq S$. 
Thus $x<v$ yields an element $z$ such that $x\leq z\in\pLoCov L v\subseteq S$. 
Therefore $u<x\leq z< v$ and $z\in S$ contradict $u\prec_S v$. Thus $[u,v]_L$ is a chain, completing the proof of Lemma \ref{lemma:svznfmsT}.
\end{proof}

\begin{lemma}\label{lemma:trunC}
For a finite lattice $L$, let $u$ and $v$ be the smallest and the largest element of $\Skel L$, respectively. Let $K$ denote the interval $[u,v]_L$. Then the lattice $K$ has the following properties: $0_K\in\Mr K$, $1_K\in\Jr K$, and $\cd K=\cd L$. 
\end{lemma}

\begin{proof} 
If we had elements $x\in \idla L u$ and $y\in L$ such that $x\parallel y$, then $u>x\wedge y\in\Mr L\subseteq \Skel L$ would contradict the fact that $u$ is the smallest element of $\Skel L$. Thus $\idla L u\subseteq \Nar L$ and $\idla L u$ is a chain.  By duality, $\fila L v\subseteq \Nar L$ and $\fila L v$ is also a chain. Collapsing the edges of these two chains, one by one, we obtain $K=[u,v]_L$. It follows from \eqref{eq:cdColLab} that $\cd K=\cd L$. 

Since $u\in\idla L u\subseteq \Nar L$, $u\notin \pLoCov L x$ for $x\in \Jr L$ and $u\notin \pUpCov L y$ for $y\in\Mr L$.
Furthermore, $u\notin \Jr L$, since otherwise $\pLoCov L u\subseteq \Skel L$ would contradict the definition of $u$. 
Therefore \eqref{eq:szlSklSubL} yields that $u\in\Mr L$.
For any $p\in\pUpCov L u$, the inclusion $\pUpCov L u\subseteq\Skel L$ implies $p\leq v$, whence $p\in[u,v]_L=K$. Thus $\pUpCov L u\subseteq \pUpCov K u$, implying that $0_K=u\in\Mr K$. 
Dually, $1_K\in\Jr K$, completing the proof of Lemma \ref{lemma:trunC}.
\end{proof}

The following lemma follows directly from Lemma \ref{lemma:svznfmsT}, so we omit its proof.

\begin{lemma}\label{lemma:skelCoordn} 
For a finite lattice $L$ with $0\in\Mr L$ and $1\in\Jr L$, let $S:=\Skel L$.  Then there exists an enumeration $\pi$ of $\Edge S$ and an $|\Edge S|$-tuple $s$ of nonnegative integers such that $L=\MExt S\pi s$.
\end{lemma}

The integer part $\max\set{k\in\ZZ: k\leq x}$ of an $x\in \RR$ is denoted by $\lfloor x\rfloor$.
For $p\in(0,1]_\RR$, we also use the notation introduced in \eqref{eq:notaRkkm} to define
\begin{equation*}
k(p):=\lfloor\log_2(1/p)\rfloor+3 \text{ \ and \ }
f(p):=2k(p)\cdot\bigl(R_{k(p)-1}(k(p))+1\bigr).
\end{equation*}

\begin{lemma}\label{lemma:skelconfine}
For each positive real number $p\leq 1$ and every finite lattice $L$, the inequality $\cd L \geq p$ implies $|\Skel L| \leq f(p)$.
\end{lemma}

\begin{proof} For the sake of contradiction, suppose the contrary, and pick a positive $p\leq 1$ and a finite lattice $L$ such that $\cd L\geq p$ but $|\Skel L|>f(p)$. 
Then $\rno L>f(p)$ by Lemma \ref{lemma:ineqLSumRS}. To ease the notation, let  $k:=k(p)$, and observe that $2^{-k}<p$.
Furthermore, let $\Sigma_1$ and $\Sigma_2$ denote the first sum, subscripted by $x\in\Jr L$, and the second sum, subscripted by $x\in\Mr L$, in \eqref{eq:SumRSdef}. The inequality $\rno L>f(p)=2k\cdot(R_{k-1}(k)+1)$ yields that $\Sigma_1>k\cdot(R_{k-1}(k)+1)$ or $\Sigma_2>k\cdot(R_{k-1}(k)+1)$. By duality, we may assume that $\Sigma_1>k\cdot(R_{k-1}(k)+1)$. 
Thus the number of summands of $\Sigma_1$ is larger than $k$ or one of the summands is larger than $R_{k-1}(k)+1$.
In the first case, $|\Jr L|>k$, whence \eqref{eq:mncgrdlcd} and $2^{-k}<p$ imply $\cd L\leq 2^{-|\Jr L|}\leq 2^{-k}<p$, a contradiction. In the second case, $\nlc x+ 1>R_{k-1}(k)+1$ for some $x\in\Jr L$, that is, 
$\nlc x >R_{k-1}(k)$. This inequality, $k\geq 3$, and the dual of Lemma \ref{lemma:ramseyMt} imply $\cd L\leq 2^{-k}<p$, a contradiction again. The proof of Lemma \ref{lemma:skelconfine} is complete.
\end{proof}

\section{The semimodular case}

\begin{lemma}\label{lemma:semimodSkelCor}
If $L$ is a finite semimodular lattice such that $0\in \Mr L$ and $1\in\Jr L$,  then $\Core L=\Core{\Skel L}$. 
\end{lemma}

\begin{proof} 
Let $S:=\Skel L$. We show first that for every edge $(u,v)$ of $S$,
\begin{equation}
\text{if }\set{u,v}\nsubseteq\Nar L\text{, then }[u,v]_L=[u,v]_S=\set{u,v}.
\label{eq:mktHhzLJvnHl}
\end{equation}

For the sake of contradiction, suppose that \eqref{eq:mktHhzLJvnHl} fails, and pick an edge $(u,v)$ of $S$ that witnesses this failure.
So $\set{u,v}\nsubseteq\Nar L$ but $[u,v]_L\neq \set{u,v}$.
Pick a minimal element $x$ from $[u,v]_L\setminus \set{u,v}$, that is, $x\in L$ with $u\prec x<v$. The non-subscripted  relational symbol $\prec$ in the proof is understood in $L$, and the covering in $S$ is denoted by $\prec_S$.

\begin{caseinproof}{Case 1} We assume that $v\notin \Nar L$. Pick an element $b\in L$ with $b\parallel v$. Clearly, $b\nleq u$. 

\begin{caseinproof}{Subcase 1a} We assume that $u\leq b$, in addition to  $b\parallel v$. In fact, $u<b$, since $u=b$ would violate $b\parallel v$. 
Since $u\leq b\wedge v<v$, $b\wedge v\in\Mr L\subseteq S$, and $u\prec_S v$, we obtain $b\wedge v=u$.  Pick an element $c\in L$ such that $u\prec c\leq b$. 
If $x$ were equal to $c$, then $u=b\wedge v\geq c\wedge x=x$ 
would be a contradiction. Thus $c$ and $x$ are distinct covers of $u$ in $L$, and so $y:=c\vee x\in\Jr L\in S$.
Since $L$ is semimodular, $x=x\vee u\preceq x\vee c=y$. 
In fact, $x\prec y$, as $x\neq y$. This covering relation and $x<v$ rule out that $v<y$. Since $y\leq v$ would lead to $c\leq b\wedge y\leq b\wedge v=u$, a contradiction, we have $y\nleq v$. Therefore $v\parallel y$, whence $v>v\wedge y\in\Mr L\subseteq S$. The membership $v\wedge y\in S$ together with $u<x\leq v\wedge y<v$ contradicts $u\prec_S v$. Consequently, Subcase 1a cannot occur.
\end{caseinproof} %End of Subcase 1a

\begin{caseinproof}{Subcase 1b} We assume that $u\parallel b$, in addition to  $b\parallel v$. Let $d:=b\wedge v$. 
Since $d\geq u$ would yield $b\geq u$, a contradiction, either $d\parallel u$ or $d<u$. 

For the sake of contradiction, suppose that $d\parallel u$. Then $d\wedge u<u$, and we may pick an element $y\in L$ such that $d\wedge u\prec y\leq d$. 
Since $u\leq y$ would lead to $u\leq d$ and $y\leq u$ to $y\leq d\wedge u$, we have $u\parallel y$. Hence $u<u\vee y\leq 
u\vee d\leq v$. Since $u\vee y\in \Jr L\subseteq S$, we obtain from $u<u\vee y\leq v$ and $u\prec_S v$ that $u\vee y=v$. Thus, by semimodularity and $d\wedge u\prec y$, we have $u=u\vee (d\wedge u)\preceq u\vee y=v$, which contradicts $u<x<v$. 

Therefore $d\nparallel u$, and we have $d<u$. We also have $d<b$ (as $b\parallel v$), whence we can pick an element $z$ with $d\prec z\leq b$. We have $z\nleq v$, since otherwise 
$z\leq b\wedge v=d$ would contradict $d\prec z$. Hence $x\vee z\nleq v$. Since $x=x\vee d \preceq x\vee z$ by semimodularity and $x<v$, we cannot have $v<x\vee z$. Thus $v\parallel x\vee z$, whence $v\wedge(x\vee z)\in\Mr L\subseteq S$. Thus $u<x\leq v\wedge(x\vee z) <v$ contradicts $u\prec_S v$ and rules out Subcase 1b.  Moreover, Case 1 is also ruled out.
\end{caseinproof} %End of Subcase 1b
\end{caseinproof}

\begin{caseinproof}{Case 2} We assume that $u\notin\Nar L$. Pick an element $e\in L$ such that $u\parallel e$. Since Case 1 is already excluded and so $v\in \Nar L$, we obtain that $e\leq v$. As $u\parallel e$, we may pick an element $w\in L$ such that $e\wedge u\prec w\leq e$. 
Since $u\leq w$ would lead to $u\leq e$ while $w\leq u$ to
$w\leq e\wedge u$, we have $u\parallel w$. Hence $u<u\vee w\in\Jr L\subseteq S$. Furthermore, $u\vee w\leq v\vee e\leq v$. Combining $u<u\vee w\leq v$ with $u\prec_S v$ and $u\vee w\in S$, we obtain that $u\vee w=v$. By semimodularity, $u=u\vee(e\wedge u)\preceq  u\vee w=v$. This contradicts $u<x<v$ and excludes Case 2. 
\end{caseinproof}

All cases of the indirect assumption have been excluded. Therefore, we have proved \eqref{eq:mktHhzLJvnHl}.

Next, we show that 
\begin{equation}
\Nar S =\Nar L\cap S.
\label{eq:narrestrtoS}
\end{equation}
Since $S$ is a sublattice of $L$, the inclusion $\Nar S \supseteq\Nar L\cap S$ is obvious. To show the converse inclusion, let $x\in \Nar S$. Then $x\in S$. 
To show that $x\in \Nar L$, assume that $y\in L$.
By Lemma \ref{lemma:svznfmsT}, there is an edge $(u,v)$ of $S$ such that $y\in[u,v]_L$. If $u=x$ or $v=x$, then $x\nparallel y$ is clear.
Hence we may assume that $x\notin \set{u,v}=[u,v]_S$. Since $x\in\Nar S$, $x\nparallel u$ and $x\nparallel v$. Hence either $x<u<v$ and $x\leq y$, or $u<v<x$ and $y\leq x$, whence $x\nparallel y$ again. This proves the converse inclusion and \eqref{eq:narrestrtoS}.

For a gluing edge $(u,v)$ of $S$, let $\set{x_1,\dots,x_{k-1}}:=[u,v]_L\setminus\set{u,v}$. Since this set is a chain by Lemma \ref{lemma:svznfmsT}, we may choose the subscripts so that $x_0:=u\prec_L x_1 \prec_L x_2 \prec_L \cdots \prec_L x_{k-1} \prec_L x_k:=v$. We claim that
\begin{equation}
[x_{i-1},x_i]\text{ is a gluing edge of }L\text{ for all }i\in\set{1,\dots,k}.
\label{eq:yTnbTsHsKnzs}
\end{equation}
To see this, it suffices to show that $x_i\in\Nar L$ for all $i\in\set{0,1,\dots,k}$. We may assume that  $i\in\set{1,\dots,k-1}$, since $x_0=u\in \Nar S\subseteq \Nar L$ and $x_k=v\in\Nar L$ by \eqref{eq:narrestrtoS}. Let $y\in L$. If $y\in [u,v]_L$, then $x_i\nparallel y$ follows from the fact that $[u,v]_L$ is a chain by Lemma \ref{lemma:svznfmsT}. Assume $y\notin [u,v]_L$. Since $y\nparallel u$ and $y\nparallel v$ by $u,v\in\Nar S\subseteq \Nar L$---see \eqref{eq:narrestrtoS}---, either $y\leq u$ or $y\geq v$, whence either $y\leq x_i$ or $y\geq x_i$, and so $x_i\nparallel y$ again. We have shown the validity of \eqref{eq:yTnbTsHsKnzs}.  

If $(u,v)$ is a non-gluing edge of $S$, then $u\notin\Nar S$ or $v\notin\Nar S$, whence \eqref{eq:narrestrtoS} implies that $\set{u,v}\nsubseteq \Nar L$ . Thus \eqref{eq:mktHhzLJvnHl} yields that for every non-gluing edge $(u,v)$ of $S$,  $[u,v]_L=[u,v]_S$ and $(u,v)$ is a non-gluing edge of $L$. This fact,  \eqref{eq:lLlflTnS}, and \eqref{eq:yTnbTsHsKnzs} yield that we obtain $L$ from $S$ in the following way: we keep the non-gluing edges of $S$, which are also kept when constructing $\Core S$ and $\Core L$, and we turn every gluing edge into one or more gluing edges, which are collapsed when forming the cores. Therefore, $\Core L=\Core{S}$, completing the proof of Lemma \ref{lemma:semimodSkelCor}.
\end{proof}

\begin{lemma}\label{lemma:SMfinite} For each positive real number $p$, the set 
\begin{equation*}
H_p:=\{\cd L: L\text{ is a finite semimodular lattice and }\cd L \geq p\}
%\label{eq:frCstnCllrK}
\end{equation*}
is finite.
\end{lemma}

\begin{proof} Assume that $0<p\in\RR$. 
For each $r\in H_p$, pick a finite semimodular lattice $L_r$ of minimum size such that $\cd{L_r}=r$. 
Observe that $0\in\Mr {L_r}$ and $1\in \Jr {L_r}$, since otherwise 
Lemma \ref{lemma:trunC}, applied to $L_r$, would yield a smaller lattice with the same congruence density.
By $r\in H_p$, we have that $\cd {L_r}=r\geq p$. Thus Lemma \ref{lemma:skelconfine} gives that $|\Skel{L_r}|\leq f(p)$.

Let $K_r:=\Core {\Skel{L_r}}$.
Since $|\Core T|\leq |T|$ for every finite lattice $T$, we have that $|K_r|\leq |\Skel{L_r}|$, whence $|K_r|\leq f(p)$. 
By Lemma \ref{lemma:semimodSkelCor}, $\Core {L_r}=K_r$. 
Thus, applying Lemma \ref{lemma:Corelem}, 
$r=\cd{L_r}=\cd{\Core{L_r}}=\cd{K_r}$.
Since $r$ was an arbitrary member of $H_p$, it follows from
$r=\cd{K_r}$ and  $|K_r|\leq f(p)$ that each element of $H_p$ is the congruence density of an at most $f(p)$-element lattice. Since there are only finitely many such lattices, we conclude the validity of Lemma \ref{lemma:SMfinite}.
\end{proof}

\section{The rest of the proofs}\label{sect:rstprfs}

We are now in the position to prove the main result.

\begin{proof}[Proof of Theorem \ref{thm:main}]
To prove part \aref{main1} by way of contradiction, suppose that it fails. 
Since $1\in \SCD$, witnessed by any finite chain, and $\SCD$ is closed under multiplication by \eqref{eq:cdgsum}, $\SCD$ is a monoid. Hence, by Lemma \ref{lemma:trunC} and the failure of \aref{main1}, there exists an infinite sequence $L_1$, $L_2$, $L_3$, \dots{} of finite lattices such that 
\begin{align}
0<p:=\cd{L_1} < \cd{L_2} < \cd{L_3} < \cd{L_4} < \dots
\label{eq:cCslMbztT}\\
\text{and, for all }n\in\Nplu,\text{ }
0_{L_n}\in\Mr{L_n}\text{ and }1_{L_n}\in\Jr{L_n}.
\label{eq:hbdtljnPlmnt}
\end{align}
We know from Lemma \ref{lemma:skelconfine} that $|\Skel{L_n}|\leq f(p)$ for all $n\in\Nplu$. Up to isomorphism, there are only finitely many at most $f(p)$-element lattices. Hence there exists a finite lattice $S$ such that $S\cong \Skel{L_n}$ holds for infinitely many $n\in\Nplu$. 
Ruling out every $L_n$ with $S\ncong \Skel{L_n}$ from the sequence \eqref{eq:cCslMbztT}, the sequence remains infinite. Therefore we may assume (after re-indexing the remaining elements) that $\Skel{L_n}=S$ for all $n\in\Nplu$.

Fix a permutation $\pi$ of $\Edge S$. 
For each $n\in \Nplu$, Lemma \ref{lemma:skelCoordn} together with \eqref{eq:hbdtljnPlmnt} yields an $|\Edge S|$-tuple $\tups n\in\Nnul^{|\Edge S|}$ such that $L_n=\MExt S \pi{\tups n}$.
By Lemma \ref{lemma:woQo}, there exist  $i,j\in\Nplu$ such that $i<j$ and $\tups i\leq \tups j$. The second part of Lemma \ref{lemma:MExt} yields that $L_j=\MExt S \pi{\tups j}$ is a dismantlable extension of $L_i=\MExt S \pi{\tups i}$.
Thus \eqref{eq:dismeXt} implies $\cd{L_i}\geq \cd{L_j}$, contradicting \eqref{eq:cCslMbztT} and thereby proving part~\aref{main1}.

\logicbreak
To prove part \aref{main2}, let $k\in\Nplu$ with $k\geq 5$.
Take the four-element Boolean lattice $B_4$, pick a permutation $\pi$ of $\Edge{B_4}$, let $s:=(k-4,0,0,0)$, and 
define $N_k:=\MExt{B_4}\pi s$. 
Note that $N_k$ does not depend on~$\pi$, that $|N_k|=k$, and that $N_5$ has its usual meaning, the five-element non-modular lattice. Let $t:=k-3$. Then $N_k$ consists of a maximal chain $0\prec a_1\prec a_2\prec \dots\prec a_t\prec 1$ and a common complement $b$ of each of $a_1$, \dots, $a_t$. 

For $i\in\set{1,\dots, t-1}$, let $\Theta_i\in\Con{N_k}$ be the congruence whose  blocks are $\set{a_1,\dots,a_i}$,  $\set{a_{i+1},\dots, a_t}$, $\set 0$, $\set b$, and $\set 1$.  Since $\Theta_1\wedge \dots \wedge \Theta_{t-1}$ is the least congruence $0_{\Con{N_k}}$ of $N_k$ and, for each $i\in\set{1,\dots, t-1}$, $N_k/\Theta_i\cong N_5$,  
\begin{equation}
N_k\text{ is a subdirect power of }N_5.
\label{eq:subdiRpow}
\end{equation}

Now let $\varW$ be a variety containing a nonmodular lattice. Since $N_5$ is a sublattice of this lattice by Dedekind's modularity criterion, $N_5\in \varW$. Thus \eqref{eq:subdiRpow} yields that for all $k\in \Nplu$ with $k\geq 8$, $N_k\in\varW$. 
We know from \cite[Theorem 2(iii)]{czg864} that
$\cd{N_k}=(8+3/2^{k-7})/64$. It is trivial, and we know from \cite{czg-lconl2} or \cite[Theorem 1(B)]{czg864}, that $\cd{B_4}=1/2$. Define 
\begin{equation}
L_{k,n}=N_k\gsum \underbrace{B_4\gsum B_4\gsum B_4\gsum\dots\gsum B_4}_{n\text{ glued-sum summands}}.
\label{eq:kSzdlKcnsrvn}
\end{equation}
Since $N_k\in\varW$ and $B_4\cong C_2\times C_2\in \varW$, we obtain from \eqref{eq:varWgsumclosed} and \eqref{eq:kSzdlKcnsrvn} that $L_{k,n}\in\varW$ for all $n\in\Nplu$ and $8\leq k\in\Nplu$. 
Applying \eqref{eq:cdgsum} to \eqref{eq:kSzdlKcnsrvn}, we obtain that 
\begin{align*}
\cd {L_{k,n}}%&=\cd{N_k}\cdot \cd{B_4}^n\cr
&= \bigl((8+3/2^{k-7})/64\bigr)\cdot (1/2)^n=2^{-(n+3)}+3/2^{n+k-1}.
\end{align*}
Hence, for each fixed $n\in\Nplu$, 
$\lim_{k\to\infty} \cd{L_{k,n}} = 2^{-(n+3)}$. Thus, 
for each $n\in\Nplu$, $2^{-(n+3)}\in \Acc{\vwSCD}$.
Therefore, $|\Acc{\vwSCD}|\geq \aleph_0$.  

To obtain the converse inequality, observe that $\vwSCD\subseteq \SCD$. Therefore, since $\SCD$ is a dually well-ordered subset of $\RR$ by part \aref{main1}, so is $\vwSCD$.
Consequently, by the dual of Lemma \ref{lemma:wLlordH}, $|\Acc{\vwSCD}|\leq |\Cls{\vwSCD}|\leq \aleph_0$. We have proved part \aref{main2} of the theorem.

\logicbreak
To prove part \aref{main3}, observe that all finite chains are semimodular. Hence  $1=\cd{C_2}\in \vsSCD$ and, furthermore, \eqref{eq:cddef2} and \eqref{eq:Alllnoa} imply $\vcd\varS L=\cd L$ for every finite $L\in\varS$.  
It is straightforward to see that whenever $K$ and $L$ are finite semimodular lattices, then so is $K\gsum L$. Thus \eqref{eq:cdgsum} implies that $\vsSCD$ is a submonoid of $\SCD$. As before, $B_4:=C_2\times C_2$. Since $1/2=\cd{B_4}\in \vsSCD$, and $\vsSCD$ is closed under products, $(1/2)^n\in\vsSCD$ for all $n\in\Nplu$. 
Thus $0=\lim_{n\to\infty}(1/2)^n\in\Acc\vsSCD$. Seeking a contradiction, suppose that $0<p\in\Acc\vsSCD$. 
Lemma \ref{lemma:SMfinite} implies that
$H_{p/2}=\{\cd L: L\in \varS$, $\cd L\geq p/2$, and $L$ is  finite$\}$ is a finite set. This contradicts $p\in\Acc\vsSCD$, proving part \aref{main3}. 
The proof of Theorem \ref{thm:main} is complete.
\end{proof}

\begin{proof}[Proof of Corollary \ref{corol:Vmod}]
Let $\varW$ be a nontrivial variety of lattices. 
It follows from $1=\cd{C_2}\in\vwSCD$, \eqref{eq:varWgsumclosed}, and \eqref{eq:cdgsum} that $\vwSCD$ is a submonoid of $\SCD$, so $\vwSCD$ is closed under products. Hence $2^{-n}=\cd{B_4}^n\in\vwSCD$ for all $n\in\Nplu$, whence $|\vwSCD|\geq \aleph_0$. 
Since $\SCD$ is dually well-ordered with respect to $\RR$ by Theorem \ref{thm:main}\aref{main1}  and   $\vwSCD\subseteq \SCD$, we obtain that
\begin{equation}
\vwSCD \text{ is also dually well-ordered with respect to }\RR.
\label{eq:KLlvKvjlHrgh}
\end{equation}
Hence the dual of 
Lemma \ref{lemma:wLlordH} yields $|\vwSCD|\leq |\Cls\vwSCD|\leq \aleph_0$. Combining this with  $|\vwSCD|\geq \aleph_0$, established earlier, we have proved part \bref{cRl1}.

The dual of Lemma \ref{lemma:wLlordH}, together with \eqref{eq:KLlvKvjlHrgh}, also yields that $\Cls\vwSCD$ is dually well-ordered with respect to $\RR$. 
Since $\vwSCD$ is a monoid  by the previous paragraph and multiplication in $\RR$ is continuous, it follows that $\Acc\vwSCD$ and $\Cls\vwSCD$ are closed under multiplication. These facts, together with $1\in\vwSCD\subseteq\Cls\vwSCD$, imply that $\Cls\vwSCD$ is a dually well-ordered monoid with respect to $\RR$, that $\Acc\vwSCD$ is a subsemigroup of it, and 
that $\vwSCD$ is a submonoid of it.
Finally, since $1/2=\cd{B_4}$ and therefore all its powers belong to $\vwSCD$, we have $0=\lim_{n\to\infty}(1/2)^n\in\Acc\vwSCD$. We have proved part \bref{cRl2}.

Since $0\in\Acc\vwSCD$, we know that $|\Acc\vwSCD|\geq 1$.  
Assume that $|\Acc\vwSCD| > 1$, and pick a number $p\in\Acc\vwSCD\setminus\set 0$.  
If $1$, which is the largest possible congruence density, belonged to $\Acc\vwSCD$, then $\vwSCD$---and hence $\Cls\vwSCD$---would contain an infinite strictly increasing sequence tending to $1$, which would contradict \bref{cRl2}.
Thus $p\notin\{0,1\}$, and the powers $p^n$ ($n\in\Nplu$) are distinct elements of the semigroup $\Acc\vwSCD$.  
Hence $|\Acc\vwSCD|\geq\aleph_0$.  
The converse inequality follows from \bref{cRl1} and $\Acc\vwSCD\subseteq\Cls\vwSCD$.  
Therefore $|\Acc\vwSCD|=\aleph_0$, and we have proved part \bref{cRl3}.

It is well known that every finite modular lattice is semimodular; see, for example, Gr\"atzer \cite[Lemma 2.21]{GGbypicture}. 
Therefore \bref{cRl4} follows from \aref{main2} and \aref{main3} of Theorem \ref{thm:main}, and the proof of Corollary \ref{corol:Vmod} is complete.
\end{proof}

\begin{proof}[Proof of Corollary \ref{corol:modL}] By Lemma \ref{corol:Vmod}\bref{cRl2},  $0\in\Acc{\Var L}$. Hence \quot{exactly one} and \quot{at most one} are equivalent, and Corollary \ref{corol:Vmod}\bref{cRl4} implies Corollary \ref{corol:modL}.
\end{proof}

\begin{proof}[Proof of Corollary \ref{corol:scnd}]
As in \eqref{eq:KLlvKvjlHrgh}, $\vwSCD$ is dually well-ordered with respect to $\RR$. Thus the dual of the second half of Lemma \ref{lemma:wLlordH} implies Corollary \ref{corol:scnd}.
\end{proof}

\begin{proof}[Proof of Corollary \ref{corol:rescvmodL}] 
Let $\varW:=\Var L$. Then $L$ is modular if and only if $\varW \subseteq \vari M$. 
By Corollary \ref{corol:Vmod}\bref{cRl1}--\bref{cRl2},  $\vwSCD$ is infinite and it is dually well-ordered with respect to $\RR$. The order type of an infinite well-ordered set $(T;\leq)$ is $\omega$ if and only if $\set{x\in T: x< u}$ is finite for every $u\in T$. 
Thus, by duality and because of $\vwSCD\subseteq (0,1]_\RR$, 
the order type of $\vwSCD$ is $\omega^\ast$ if and only if 
\begin{equation}
\text{for every }u\in\vwSCD\text{, } (u,1]_\RR\cap \vwSCD
\text{ is finite.}
\label{eq:fbthlMvGs}
\end{equation}
Therefore, by Corollary \ref{corol:Vmod}\bref{cRl4}, we need to show that $\vwSCD$ has exactly one accumulation point if and only if \eqref{eq:fbthlMvGs} holds.
Equivalently, since $0\in\Acc\vwSCD$ by Lemma \ref{corol:Vmod}\bref{cRl2},  it suffices to show that 
\begin{equation}
\vwSCD\text{ has a positive accumulation point}\iff\text{\eqref{eq:fbthlMvGs} fails}.
\label{eq:mKrFlBlLCRkFtkn}
\end{equation}
To verify the \quot{$\Leftarrow$} part of \eqref{eq:mKrFlBlLCRkFtkn}, assume that \eqref{eq:fbthlMvGs} fails. 
Pick an element $u\in\vwSCD$ such that $(u,1]_\RR\cap \vwSCD$ is infinite. 
Then $[u,1]_\RR$ contains infinitely many elements from $\vwSCD$, whence it contains an accumulation point $d$ of $\vwSCD$. Since $u\in\vwSCD$, we have that $0<u\leq d$. 
Thus  $\vwSCD$ has a positive accumulation point. This proves the \quot{$\Leftarrow$} part of \eqref{eq:mKrFlBlLCRkFtkn}. 

To verify the \quot{$\Rightarrow$} part, let $p$ be a positive accumulation point of $\vwSCD$. 
By Corollary \ref{corol:scnd}, $p$ is a right accumulation point. Hence $(p,1]_\RR\cap \vwSCD$ is infinite.
Since $0$ is also an accumulation point of $\vwSCD$, there exists an element $u\in\vwSCD$ such that $0<u<p$.
As $(p,1]_\RR\cap \vwSCD\subseteq (u,1]_\RR\cap \vwSCD$, the set $(u,1]_\RR\cap \vwSCD$ is also infinite. Hence  \eqref{eq:fbthlMvGs} fails, and we have shown the \quot{$\Rightarrow$} part of \eqref{eq:mKrFlBlLCRkFtkn}.  Thus \eqref{eq:mKrFlBlLCRkFtkn} holds, and the proof of Corollary \ref{corol:rescvmodL} is complete.
\end{proof}

\end{document}